\documentclass{amsart}

\usepackage{amssymb}
\usepackage{amsmath,amsthm}
\usepackage{hyperref}
\usepackage{tikz}

\newtheorem{theorem}{Theorem}[section]

\newtheorem{lemma}[theorem]{Lemma}

\theoremstyle{definition}

\theoremstyle{remark}

\DeclareMathOperator{\Cay}{Cay}

\newcommand{\ba}{\mathbf{a}}
\newcommand{\bb}{\mathbf{b}}
\newcommand{\be}{\mathbf{0}}

\newcommand{\GH}{\hat{G}}
\newcommand{\HH}{\hat{H}}
\newcommand{\gH}{\hat{g}}
\newcommand{\SH}{\hat{S}}

\newcommand{\ZZ}{\mathbb{Z}}
\newcommand{\V}{\mathrm{V}}
\newcommand{\D}{\mathrm{D}}
\newcommand{\K}{\mathrm{K}}
\newcommand{\OO}{\mathrm{O}}

\begin{document}

\title{Counterexamples to ``A Conjecture on Induced Subgraphs of Cayley Graphs''}

\author{Florian Lehner}
\author{Gabriel Verret}
\address{Florian Lehner, Institute of Discrete Mathematics, Graz University of Technology, Steyrergasse 30, 8010 Graz, Austria.}
\email{mail@florian-lehner.net}
\address{Gabriel Verret, Department of Mathematics, The University of Auckland, Private Bag 92019, Auckland 1142, New Zealand.}
\email{g.verret@auckland.ac.nz}



\begin{abstract}
Recently, Huang gave a very elegant proof of the Sensitivity Conjecture by proving that hypercube graphs have the following property: every induced subgraph on a set of more than half its vertices  has maximum degree at least $\sqrt{d}$, where $d$ is the valency of the hypercube.  This was generalised by Alon and Zheng who proved that every Cayley graph on an elementary abelian $2$-group has the same property. Very recently, Potechin and Tsang proved an analogous results for Cayley graphs on abelian groups. They also conjectured that all Cayley graphs have the analogous property. We disprove this conjecture by constructing various counterexamples, including an infinite family of Cayley graphs of unbounded valency which admit an induced subgraph of maximum valency $1$ on a set of more than half its vertices.
\end{abstract}

\maketitle

\section{Introduction}

All graphs and groups in this paper are finite. Recently, Huang~\cite{Huang} gave a very elegant proof of the Sensitivity Conjecture~\cite{Nisan}  by proving that hypercube graphs have the following property: every induced subgraph on a set of more than half its vertices  has maximum degree at least $\sqrt{d}$, where $d$ is the valency of the hypercube. This is best possible, as shown by Chung, F\"{u}redi, Graham and Seymour~\cite{Chung}. This was generalised by Alon and Zheng who proved that every Cayley graph on an elementary abelian $2$-group has the same property~\cite{Alon}. In their concluding remarks, they point out that this result cannot generalise directly to other groups, but that it would be interesting to investigate the possible analogs for Cayley graphs on other groups. 

Very recently, Potechin and Tsang proved such an analogous result for Cayley graphs on all abelian groups~\cite{Potechin}, by replacing the $\sqrt{d}$ bound by $\sqrt{d/2}$. (More precisely, they prove their result with the bound $\sqrt{x+x'/2}$, where $x$ is the number of involutions in the connection set of the Cayley graph, and $x'$ the number of non-involutions.) They also conjectured that their result should hold for all Cayley graphs~\cite[Conjecture~1]{Potechin} and asked whether even all vertex-transitive graphs might have this property. 

In this short note, we give three infinite families of vertex-transitive graph such that every graph in these families admits an induced  subgraph of maximum valency $1$ on a set of more than half its vertices. First is the well-known family of odd graphs. Note that this family has unbounded valency and so these graphs fail to have the required property in a very strong sense. On the other hand, they are not Cayley. The second family consists of some $3$-regular Cayley graphs on dihedral groups. The last family consists of an infinite family of graphs of unbounded valency which are Cayley on groups defined via iterated wreath products.  Both the latter families are thus counterexamples to~\cite[Conjecture~1]{Potechin}. (We also note that, in our first two families, the induced subgraph in question is a matching, that is, each vertex has valency $1$.)

\section{Odd graphs}
For $n\geq 1$, the  \emph{odd graph}  $\OO_n$ has vertex-set the $n$-subsets of a $(2n+1)$-set $\Omega$, with two vertices adjacent if and only if the corresponding subsets are disjoint. For example $\OO_1\cong \K_3$, while $\OO_2$ is isomorphic to the Petersen graph. Note that there is an obvious  action of $S_{2n+1}$ as a vertex-transitive group of automorphism of $\OO_n$, so these graphs are vertex-transitive. On the other hand, these graphs are not Cayley graphs for $n\geq 2$~\cite{Godsil}. It is easy to check that $\OO_n$ is $(n+1)$-regular. Let $\omega\in\Omega$ and let $U$ be the set of $n$-subsets of $\Omega$ that do not contain $\omega$. Note that 
$$\frac{|U|}{|\V(\OO_n)|}=\frac{\binom{2n}{n}}{\binom{2n+1}{n}}=\frac{n+1}{2n+1}>\frac{1}{2},$$
 but the induced subgraph on $U$ is $1$-regular. 

\section{$3$-valent Cayley graphs on dihedral groups} \label{sec:dihedral}
For a group $G$ and an inverse-closed and identity-free subset $S$ of $G$, the \emph{Cayley graph} $\Cay(G,S)$ is the graph with vertex-set $G$ and two vertices $g$ and $h$ adjacent if and only $g^{-1}h\in S$.  For $n\geq 1$, we denote by  $\D_{2n}$ the dihedral group $\langle a,b \mid a^n=b^2=(ab)^2=1\rangle$ of order $2n$. 

Let $\Gamma_{18}=\Cay(\D_{18},\{b,ab,a^3b\})$. There is a set $U_{18}$ of $10$ vertices of $\Gamma_{18}$ such that the induced subgraph on $U_{18}$ is $1$-regular,  as can be seen on the picture below, where the elements of $U_{18}$ are coloured gray.

\tikzstyle{white}=[circle,draw,fill=white, minimum size=9mm]
\tikzstyle{black}=[circle,draw,fill=gray, minimum size=9mm]

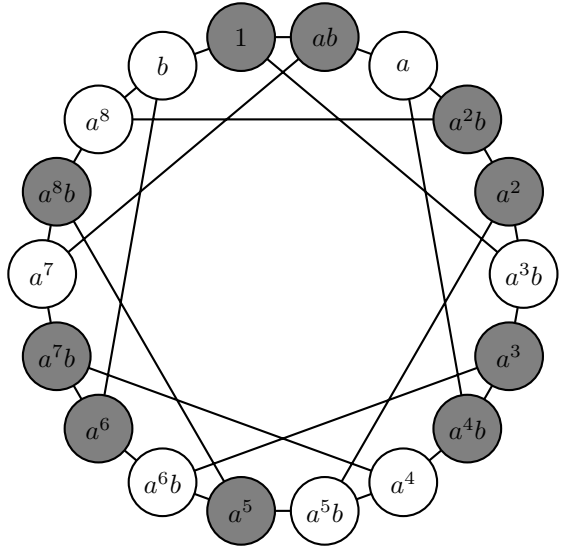
\begin{figure}[hhh]
\begin{tikzpicture}[thick,scale=0.8]
	\draw \foreach \x in {0,1,...,8}
	{ 
		(40*\x:4)   -- (40*\x+100:4) 
		(40*\x+20:4)   -- (40*\x+40:4) 
		(40*\x:4)  -- (40*\x+20:4) 
	};
	\node at (100:4)[black]  {$1$};
	\node at (140:4) [white]{$a^8$};
	\node at (180:4) [white]{$a^7$};
	\node at (220:4) [black]{$a^6$};
	\node at (260:4) [black]{$a^5$};
	\node at (300:4)[white] {$a^4$};
	\node at (340:4) [black]{$a^3$};
	\node at (20:4) [black]{$a^2$};
	\node at (60:4)[white] {$a$};
	
	\node at (120:4) [white]{$b$};
	\node at (160:4) [black]{$a^8b$};
	\node at (200:4) [black]{$a^7b$};
	\node at (240:4) [white] {$a^6b$};
	\node at (280:4) [white]{$a^5b$};
	\node at (320:4) [black]{$a^4b$};
	\node at (0:4)   [white]{$a^3b$};
	\node at (40:4) [black] {$a^2b$};
	\node at (80:4) [black] {$ab$};

\end{tikzpicture}
\caption{$\Gamma_{18}=\Cay(\D_{18},\{b,ab,a^3b\})$}
\end{figure}

Starting from $\Gamma_{18}$, it is easy to get an infinite family of further examples using covers: for $m\geq 1$, let $\Gamma_{18m}=\Cay(\D_{18m},\{b,ab,a^3b\})$ and let $N$ be the subgroup of $\D_{18m}$ generated by $a^9$. Note that $N$ is a normal subgroup of $\D_{18m}$ of order $m$ and $\D_{18m}/N\cong \D_{18}$. Let $\varphi:\D_{18m}\to \D_{18}$ be the natural projection and let $U_{18m}$ be the preimage of $U_{18}$. Now, $|U_{18m}|=10m$ and, since $\Gamma_{18m}$ is a (normal) cover of $\Gamma_{18}$, the induced subgraph on $U_{18m}$ is $1$-regular. 

\section{Cayley graphs on iterated wreath products}\label{sec:wreath}

Let $\ZZ_2=\{0,1\}$ denote the cyclic group of order $2$, let $G$ be a group with identity element $1_G$ and let $(\ZZ_2)^G$ denote the set of functions from $G$ to $\ZZ_2$. Note that $(\ZZ_2)^G$ forms a group under pointwise addition. Let $\be$ be the identity element of $(\ZZ_2)^G$ (that is, the function mapping every element of $G$ to $0$). Note also that there is natural action of $G$ on $(\ZZ_2)^G$. Written in exponential notation, we have that if $\ba\in (\ZZ_2)^G$ and $g \in G$, then $\ba^g$ is the element of $(\ZZ_2)^G$ defined by  $\ba^g(x)=\ba(g^{-1}x)$ for every $x\in G$. 

 The  \emph{wreath product} $\ZZ_2 \wr G$ is the group consisting of all pairs $(\ba,g)$ where $\ba\in (\ZZ_2)^G$ and $g \in G$, with the group operation given by 
\[(\ba,g) (\bb,h) = (\ba+\bb^g,g  h ).\]

For $g\in G$, let  $\ba_g\in (\ZZ_2)^G$ be the function mapping $g$ to $1$ and every other element of $G$ to $0$. If $S$ is a generating set for $G$, then
\[\{(\ba_1,1_G)\} \cup\{(\be, s) \mid s \in S\}\]
is a generating set for $\ZZ_2 \wr G$ which we call the \emph{canonical} generating set for $\ZZ_2 \wr G$ (with respect to $S$).

\begin{lemma}\label{lemma:wreath}
Let $S$ be a generating set for a group $G$, let  $\GH=\ZZ_2 \wr G$ and let $\SH$ be the canonical generating set for $\GH$ with respect to  $S$.
If $\Cay(G,S)$ is bipartite and has an induced subgraph of maximum degree $1$ on more than half its vertices, then the same is true for $\Cay(\GH, \SH)$.
\end{lemma}

\begin{proof}
We first show that $\Cay(\GH, \SH)$ is bipartite. Call an element of $G$ \emph{even} if it lies in the same part of the bipartition of $\Cay(G,S)$ as $1_G$, and \emph{odd} otherwise. Call an element $\ba$ of  $(\ZZ_2)^G$ \emph{even} if $\ba$ maps an even number of elements of $G$ to $1$, and call $\ba$ \emph{odd} otherwise. Finally, call an element $\gH = (\ba, g) \in \GH$ \emph{even} if  $\ba$ and $g$  are either both even or both odd, and call $\gH$ \emph{odd} otherwise. It is straightforward to check that if $\hat s \in \SH$, then $\gH$ is even if and only if $\gH \hat s$ is odd. Thus the partition of $\GH$ into even and odd elements is a bipartition of $\Cay(\GH, \SH)$. 

Let $H \subseteq G$ be such that $|H| > |G|/2$ and the subgraph of $\Cay(G,S)$ induced by $H$ has maximum degree $1$. Denote by $G_{\text{even}}$ and $G_{\text{odd}}$ the set of even and odd elements of $G$, respectively. For each $\ba\in(\ZZ_2)^G$, let $[\ba] = \{(\ba, g)\mid g \in G\}\subseteq \GH$ and define a subset $H_{\ba}$ of $G$ as follows:
\[
    H_{\ba} = 
    \begin{cases}
        H & \text{if } \ba = \be,\\
        G_{\text{odd}}& \text{if } \ba=\ba_g \text{ for some $g\in G_{\text{even}}$},\\
        G_{\text{even}}& \text{otherwise.}\\
    \end{cases}
\]

Let $\HH= \{(\ba, h) \mid \ba\in(\ZZ_2)^G, h \in H_{\ba}\}\subseteq \GH$. Clearly, $|\HH \cap [\be]| = |H| > |G|/2$ while, for $\ba \neq \be$, we have  $|\HH \cap [\ba]| = |G|/2$. It follows that  $|\HH| > |\GH| /2$. Let $(\ba, h)\in \HH$. We show that $(\ba, h)$ has at most one neighbour in $\HH$. By definition, $h\in H_{\ba}$. 

If $\ba=\be$, then $h\in H$. Note that, if $g\in G_{\text{even}}$, then $H_{\ba_g}=G_{\text{odd}}$, whereas if $g\in G_{\text{odd}}$, then $H_{\ba_g}=G_{\text{even}}$. In particular, for every $g\in G$, $(\ba_g,g)\notin \HH$. It follows that $(\be,h)(\ba_1,1_G)=(\ba_1^h,h)=(\ba_h,h)\notin \HH$. On the other hand, for $s\in S$, $(\be, h)(\be,s)=(\be,hs)\in\HH$ if and only if $hs\in H$. This shows that  the number of neighbours of $(\be, h)$ in $\HH$ is the same as the number of neighbours of $h$ in $H$ and therefore at most $1$ (with equality reached for some $h\in H$). 

Assume now $\ba \neq \be$. Since the partition of $\GH$ into even and odd elements is a bipartition of $\Cay(\GH, \SH)$, it follows from the definition of $H_\ba$ that there are no two adjacent elements in $\HH \cap [\ba]$. Since $(\ba, h)$ has exactly one neighbour outside of $[\ba]$ (namely $(\ba, h)(\ba_1,1_G)$), it follows that $(\ba, h)$ has at most one neighbour in $\HH$. This concludes the proof that the subgraph induced by $\HH$ has maximum degree $1$.
\end{proof}

By starting with, say $(G,S)=(\ZZ_2,\{1\})$ and applying Lemma~\ref{lemma:wreath} repeatedly, one obtains an infinite family of Cayley graphs of unbounded valency such that every graph in the family  admits an induced  subgraph of maximum degree $1$ on a set of more than half its vertices, as claimed.

\section{Concluding remarks}

\begin{enumerate}
\item Recall that a \emph{covering map} $f$ from a graph $\hat\Gamma$ to a graph $\Gamma$ is a surjective map that is a local isomorphism. If such a map exists, then $\hat\Gamma$ is a  \emph{covering graph}  of  $\Gamma$. It is easy to see that if $\Gamma$ is a $d$-regular graph admitting an induced  subgraph of maximum degree $1$ on a set of more than half its vertices, then $\hat\Gamma$ has the same property. (It is well known that all the vertex-fibers have the same cardinality, so one can simply take the preimage of the set of more than half the vertices of $\Gamma$.) Starting from one example, one can thus construct infinitely many having the same valency, vertex-transitivity, Cayleyness, etc.

\item As a consequence of the above remark, we can construct infinite families of $d$-regular Cayley graphs of order $n$ admitting induced subgraphs of maximum degree $1$ on $\frac{1+\varepsilon(d)}2 n$ vertices.

\item It is an easy exercise that if   $\Gamma$ is a $d$-regular graph of order $n$ admitting a subset $X$ of vertices such that the induced graph on $X$ has maximum valency at most $1$, then $|X|/n\leq \frac{d}{2d-1}$.  Note that the Odd graph $\OO_{d-1}$ attains this bound and so is extremal from this perspective. It would be interesting to know if this bound can be achieved by Cayley graphs.

\item In light of the results from \cite{Alon,Potechin} on Cayley graphs of abelian groups, it seems natural ask if there is a natural family of nonabelian groups having the same property. In Sections~\ref{sec:dihedral} and~\ref{sec:wreath}, we give examples of Cayley graphs on some dihedral groups and some $2$-groups that do not have this property. Given that both these families of groups are in some sense close to being abelian (dihedral groups have a cyclic subgroup of index $2$, while $2$-groups are nilpotent), the question of determining whether any natural family of nonabelian groups has this property seems even more interesting.


\end{enumerate}

\section*{Acknowledgements}
Florian Lehner acknowledges the support of the Austrian Science Fund (FWF), grant J 3850-N32 and grant P 31889-N35. Gabriel Verret is grateful to the N.Z. Marsden Fund for its support  (via grant UOA1824).

\end{document}